\documentclass[a4paper,12pt]{amsart}
\usepackage[width=0.8\paperwidth,height=0.85\paperheight,twoside=false,includehead,includefoot]{geometry}
\newtheorem{thm}{Theorem}[section]
\newtheorem{prop}[thm]{Proposition}
\newtheorem{lem}[thm]{Lemma}
\theoremstyle{definition}
\newtheorem{defn}[thm]{Definition}
\newcommand{\A}{\mathcal{A}}
\newcommand{\Q}{\mathbb{Q}}
\newcommand{\Z}{\mathbb{Z}}
\begin{document}
\title{Sum formula for finite multiple zeta values}
\author{Shingo Saito}
\address{Faculty of Arts and Science, Kyushu University,
 744, Motooka, Nishi-ku, Fukuoka, 819-0395, Japan}
\email{ssaito@artsci.kyushu-u.ac.jp}
\author{Noriko Wakabayashi}
\address{Faculty of Engineering, Kyushu Sangyo University,
 3-1, Matsukadai 2-chome, Higashi-ku, Fukuoka, 813-8503, Japan}
\email{noriko@ip.kyusan-u.ac.jp}
\keywords{Finite multiple zeta value, sum formula}
\subjclass[2010]{Primary 11M32; Secondary 05A19}
\begin{abstract}
 The sum formula is one of the most well-known relations among multiple zeta values.
 This paper proves a conjecture of Kaneko
 predicting that an analogous formula holds for finite multiple zeta values.
\end{abstract}
\maketitle

\section{Introduction}
\subsection{Finite multiple zeta values}
The \emph{multiple zeta values} (MZVs) and \emph{multiple zeta-star values} (MZSVs)
are defined by
\begin{align*}
 \zeta(k_1,\dots,k_n)
 &=\sum_{m_1>\dots>m_n\ge1}\frac{1}{m_1^{k_1}\dotsm m_n^{k_n}},\\
 \zeta^{\star}(k_1,\dots,k_n)
 &=\sum_{m_1\ge\dots\ge m_n\ge1}\frac{1}{m_1^{k_1}\dotsm m_n^{k_n}}
\end{align*}
for $k_1,\dots,k_n\in\Z_{\ge1}$ with $k_1\ge2$.
They are both generalizations of the Riemann zeta values $\zeta(k)$
at positive integers.

Among a large number of variants of the MZ(S)Vs,
there has recently been growing interest
in \emph{finite multiple zeta(-star) values} (FMZ(S)Vs).
Set $\A=(\prod_{p}\Z/p\Z)/(\bigoplus_{p}\Z/p\Z)$,
where $p$ runs over all primes;
in other words, the elements of $\A$ are of the form $(a_p)_p$, where $a_p\in\Z/p\Z$,
and two elements $(a_p)$ and $(b_p)$ are identified if and only if
$a_p=b_p$ for all but finitely many primes $p$.
We shall simply write $a_p$ for $(a_p)$ since no confusion is likely.
The following definition is due to Zagier (see~\cite{KZ}):
\begin{defn}
 For $k_1,\dots,k_n\in\Z_{\ge1}$, we define
 \begin{align*}
  \zeta_{\A}(k_1,\dots,k_n)
  &=\sum_{p>m_1>\dots>m_n\ge1}\frac{1}{m_1^{k_1}\dotsm m_n^{k_n}}\in \A,\\
  \zeta_{\A}^{\star}(k_1,\dots,k_n)
  &=\sum_{p>m_1\ge\dots\ge m_n\ge1}\frac{1}{m_1^{k_1}\dotsm m_n^{k_n}}\in \A
 \end{align*}
 and call them \emph{finite multiple zeta(-star) values}.
\end{defn}

We spell out two easy properties of FMZ(S)Vs that will be used later;
for the proofs and more properties,
see \cite{Hoffman,KZ,Zhao} and the introduction of \cite{SW}.

\begin{prop}\label{prop:basic}
 \begin{enumerate}
  \item We have $\zeta_{\A}(k)=0$ for all $k\in\Z_{\ge1}$.
  \item For $k_1,k_2\in\Z_{\ge1}$, we have
   \[
    \zeta_{\A}(k_1,k_2)=\zeta_{\A}^{\star}(k_1,k_2)
    =(-1)^{k_1}\binom{k_1+k_2}{k_1}\frac{B_{p-k_1-k_2}}{k_1+k_2}.
   \]
 \end{enumerate}
\end{prop}

Here the numbers $B_m$ are the Bernoulli numbers given by
\[
 \sum_{m=0}^{\infty}B_m\frac{x^m}{m!}=\frac{x}{1-e^{-x}}\in\Q[[x]].
\]

\subsection{Sum formula}
The sum formula is a basic class of relations among MZ(S)Vs
and has been generalized in various directions.
For $k,n\in\Z$ with $1\le n\le k-1$, set
\[
 I_{k,n}=\{(k_1,\dots,k_n)\in\Z_{\ge1}^n\mid k_1+\dots+k_n=k,\;k_1\ge2\}.
\]

\begin{thm}[Sum formula \cite{Granville,Hoffman92}]
 For $k,n\in\Z$ with $1\le n\le k-1$, we have
 \begin{align*}
  \sum_{(k_1,\dots,k_n)\in I_{k,n}}\zeta(k_1,\dots,k_n)&=\zeta(k),\\
  \sum_{(k_1,\dots,k_n)\in I_{k,n}}\zeta^{\star}(k_1,\dots,k_n)
  &=\binom{k-1}{n-1}\zeta(k).
 \end{align*}
\end{thm}

Kaneko \cite{Kaneko} conjectured the following analogous relations for FMZ(S)Vs:
\begin{align*}
 \sum_{(k_1,\dots,k_n)\in I_{k,n}}\zeta_{\A}(k_1,\dots,k_n)
 &=\biggl(1+(-1)^n\binom{k-1}{n-1}\biggr)\frac{B_{p-k}}{k},\\
 \sum_{(k_1,\dots,k_n)\in I_{k,n}}\zeta_{\A}^{\star}(k_1,\dots,k_n)
 &=\biggl((-1)^n+\binom{k-1}{n-1}\biggr)\frac{B_{p-k}}{k}.
\end{align*}
The aim of this paper is to prove the conjecture and its generalizations given below.

For $k,n,i\in\Z$ with $1\le i\le n\le k-1$, set
\[
 I_{k,n,i}=\{(k_1,\dots,k_n)\in\Z_{\ge1}^n\mid k_1+\dots+k_n=k,\;k_i\ge2\};
\]
note that $I_{k,n,1}=I_{k,n}$.

\begin{thm}[Main theorem]\label{thm:main}
 For $k,n,i\in\Z$ with $1\le i\le n\le k-1$, we have
 \begin{align*}
  \sum_{(k_1,\dots,k_n)\in I_{k,n,i}}\zeta_{\A}(k_1,\dots,k_n)
  &=(-1)^{i-1}\biggl(\binom{k-1}{i-1}+(-1)^n\binom{k-1}{n-i}\biggr)\frac{B_{p-k}}{k},\\
  \sum_{(k_1,\dots,k_n)\in I_{k,n,i}}\zeta_{\A}^{\star}(k_1,\dots,k_n)
  &=(-1)^{i-1}\biggl((-1)^n\binom{k-1}{i-1}+\binom{k-1}{n-i}\biggr)\frac{B_{p-k}}{k}.
 \end{align*}
\end{thm}

Setting $i=1$ gives Kaneko's conjecture.

\section{Proof of the main theorem}
For notational simplicity, we write the sums to be computed as
\[
  S_{k,n,i}=\sum_{(k_1,\dots,k_n)\in I_{k,n,i}}\zeta_{\A}(k_1,\dots,k_n),\qquad
  S_{k,n,i}^{\star}=\sum_{(k_1,\dots,k_n)\in I_{k,n,i}}\zeta_{\A}^{\star}(k_1,\dots,k_n)
\]
for $k,n,i\in\Z$ with $1\le i\le n\le k-1$.

\subsection{Recurrence relations}
We begin the proof by establishing recurrence relations
for $S_{k,n,i}$ and $S_{k,n,i}^{\star}$.
We will show the recurrence relations
by expressing products of FMZ(S)Vs as sums of FMZ(S)Vs
via the \emph{harmonic product} (see~\cite{Hoffman_algebra}).
Since explaining the harmonic product in its full generality
is unnecessarily cumbersome, we shall only illustrate it by examples.
If $k_1,k_2,l\in\Z_{\ge1}$, then Proposition~\ref{prop:basic} (1) shows that
\begin{align*}
 0&=\zeta_{\A}(k_1,k_2)\zeta_{\A}(l)\\
 &=\Biggl(\sum_{m_1>m_2}\frac{1}{m_1^{k_1}m_2^{k_2}}\Biggr)
  \Biggl(\sum_{m}\frac{1}{m^l}\Biggr)\\
 &=\Biggl(\sum_{m>m_1>m_2}+\sum_{m_1>m>m_2}+\sum_{m_1>m_2>m}
   +\sum_{m_1=m>m_2}+\sum_{m_1>m_2=m}\Biggr)\frac{1}{m_1^{k_1}m_2^{k_2}m^l}\\
 &=\zeta_{\A}(l,k_1,k_2)+\zeta_{\A}(k_1,l,k_2)+\zeta_{\A}(k_1,k_2,l)+\zeta_{\A}(k_1+l,k_2)+\zeta_{\A}(k_1,k_2+l),
\end{align*}
where $m_1$, $m_2$, and $m$ are all assumed to be positive integers less than $p$,
and similarly that
\[
 0=\zeta_{\A}^{\star}(l,k_1,k_2)+\zeta_{\A}^{\star}(k_1,l,k_2)+\zeta_{\A}^{\star}(k_1,k_2,l)
   -\zeta_{\A}^{\star}(k_1+l,k_2)-\zeta_{\A}^{\star}(k_1,k_2+l).
\]
An analogous procedure leads to the following lemma:
\begin{lem}\label{lem:recurrence_each}
 For $n\in\Z_{\ge2}$ and $k_1,\dots,k_{n-1},l\in\Z_{\ge1}$, we have
 \begin{align*}
  \sum_{j=1}^{n}\zeta_{\A}(k_1,\dots,k_{j-1},l,k_j,\dots,k_{n-1})
  +\sum_{j=1}^{n-1}\zeta_{\A}(k_1,\dots,k_{j-1},k_j+l,k_{j+1},\dots,k_{n-1})&=0,\\
  \sum_{j=1}^{n}\zeta_{\A}^{\star}(k_1,\dots,k_{j-1},l,k_j,\dots,k_{n-1})
  -\sum_{j=1}^{n-1}\zeta_{\A}^{\star}(k_1,\dots,k_{j-1},k_j+l,k_{j+1},\dots,k_{n-1})&=0.
 \end{align*}
\end{lem}

\begin{proof}
 Expand the left-hand sides of
 $\zeta_{\A}(k_1,\dots,k_{n-1})\zeta_{\A}(l)=0$ and
 $\zeta_{\A}^{\star}(k_1,\dots,k_{n-1})\zeta_{\A}^{\star}(l)=0$.
\end{proof}

\begin{prop}[Recurrence relations]\label{prop:recurrence}
 For $k,n,i\in\Z$ with $2\le i+1\le n\le k-1$, we have
 \begin{align*}
  (n-i)S_{k,n,i}+iS_{k,n,i+1}+(k-n)S_{k,n-1,i}&=0,\\
  (n-i)S_{k,n,i}^{\star}+iS_{k,n,i+1}^{\star}-(k-n)S_{k,n-1,i}^{\star}&=0.
 \end{align*}
\end{prop}

\begin{proof}
 Summing the equations in Lemma~\ref{lem:recurrence_each}
 over all $(k_1,\dots,k_{n-1},l)\in I_{k,n,i}$
 gives the desired recurrence relations.
 Indeed, the map
 \[
  (k_1,\dots,k_{n-1},l)\mapsto(k_1,\dots,k_{j-1},l,k_j,\dots,k_{n-1})
 \]
 defined on $I_{k,n,i}$ is a bijection onto $I_{k,n,i+1}$ for $j=1,\dots,i$
 and onto $I_{k,n,i}$ for $j=i+1,\dots,n$;
 under the map
 \[
  \bigl((k_1,\dots,k_{n-1},l),j\bigr)\mapsto(k_1,\dots,k_{j-1},k_j+l,k_{j+1},\dots,k_{n-1})
 \]
 from $I_{k,n,i}\times\{1,\dots,n-1\}$ to $I_{k,n-1,i}$,
 the preimage of each $(k_1',\dots,k_{n-1}')\in I_{k,n-1,i}$ is of cardinality
 \[
  \sum_{\substack{1\le j\le n-1\\j\ne i}}(k_j'-1)+(k_i'-2)=k-n.\qedhere
 \]
\end{proof}

\subsection{Computation of $S_{k,n,i}^{\star}$}
\begin{lem}[Initial values]\label{lem:sum_zeta-star_n=k-1}
 For $k,i\in\Z$ with $1\le i\le k-1$, we have
 \[
  S_{k,k-1,i}^{\star}=(-1)^{i-1}\binom{k}{i}\frac{B_{p-k}}{k}.
 \]
\end{lem}

\begin{proof}
 By the duality theorem for FMZSVs~\cite[Theorem~4.6]{Hoffman}
 and Proposition~\ref{prop:basic} (2), we find that
 \[
  S_{k,k-1,i}^{\star}
  =\zeta_{\A}^{\star}(\underbrace{1,\dots,1}_{i-1},2,\underbrace{1,\dots,1}_{k-i-1})
  =-\zeta_{\A}^{\star}(i,k-i)
  =(-1)^{i-1}\binom{k}{i}\frac{B_{p-k}}{k}.\qedhere
 \]
\end{proof}

\begin{prop}\label{prop:sum_zeta-star}
 For $k,n,i\in\Z$ with $1\le i\le n\le k-1$, we have
 \[
  S_{k,n,i}^{\star}
  =(-1)^{i-1}\biggl((-1)^n\binom{k-1}{i-1}+\binom{k-1}{n-i}\biggr)\frac{B_{p-k}}{k}.
 \]
\end{prop}

\begin{proof}
 The proof is by backward induction on $n$.

 We first consider the case $n=k-1$.
 If $k$ is even,
 then the identity trivially follows from Lemma~\ref{lem:sum_zeta-star_n=k-1}
 because $B_{p-k}=0$ (in $\Q$ and so in $\Z/p\Z$ as well) whenever
 $p$ is a prime at least $k+3$.
 If $k$ is odd, then the identity again follows
 from Lemma~\ref{lem:sum_zeta-star_n=k-1} because
 \[
  (-1)^n\binom{k-1}{i-1}+\binom{k-1}{n-i}
  =\binom{k-1}{i-1}+\binom{k-1}{k-i-1}=\binom{k}{i}.
 \]

 Now assume that the identity holds for $n$.
 Then Proposition~\ref{prop:recurrence} shows that
 \begin{align*}
  (k-n)S_{k,n-1,i}^{\star}
  &=(n-i)S_{k,n,i}^{\star}+iS_{k,n,i+1}^{\star}\\
  &=(n-i)(-1)^{i-1}\biggl((-1)^n\binom{k-1}{i-1}+\binom{k-1}{n-i}\biggr)\frac{B_{p-k}}{k}\\
  &\qquad+i(-1)^i\biggl((-1)^n\binom{k-1}{i}+\binom{k-1}{n-i-1}\biggr)\frac{B_{p-k}}{k}\\
  &=(-1)^{i-1}\biggl((n-i)(-1)^n\binom{k-1}{i-1}+(k-n+i)\binom{k-1}{n-i-1}\biggr)
    \frac{B_{p-k}}{k}\\
  &\qquad+(-1)^i\biggl((k-i)(-1)^n\binom{k-1}{i-1}+i\binom{k-1}{n-i-1}\biggr)
   \frac{B_{p-k}}{k}\\
  &=(k-n)(-1)^{i-1}\biggl((-1)^{n-1}\binom{k-1}{i-1}+\binom{k-1}{n-i-1}\biggr)
    \frac{B_{p-k}}{k}.
 \end{align*}
 Therefore the identity holds for $n-1$ as well and the proof is complete.
\end{proof}

\subsection{Computation of $S_{k,n,i}$}
Observe that each (F)MZV can be written as a $\Z$-linear combination of (F)MZSVs
and vice versa, an example being
\begin{align*}
 \zeta_{\A}(k_1,k_2,k_3)
 &=\sum_{m_1>m_2>m_3}\frac{1}{m_1^{k_1}m_2^{k_2}m_3^{k_3}}\\
 &=\Biggl(\sum_{m_1\ge m_2\ge m_3}-\sum_{m_1=m_2\ge m_3}
   -\sum_{m_1\ge m_2=m_3}+\sum_{m_1=m_2=m_3}\Biggr)\frac{1}{m_1^{k_1}m_2^{k_2}m_3^{k_3}}\\
 &=\zeta_{\A}^{\star}(k_1,k_2,k_3)-\zeta_{\A}^{\star}(k_1+k_2,k_3)
   -\zeta_{\A}^{\star}(k_1,k_2+k_3)+\zeta_{\A}^{\star}(k_1+k_2+k_3),
\end{align*}
where $m_1$, $m_2$, and $m_3$ are all assumed to be positive integers less than $p$.

\begin{lem}\label{lem:S_S-star}
 For $k,n\in\Z$ with $1\le n\le k-1$, we have
 \[
  S_{k,n,1}=\sum_{j=0}^{n-1}(-1)^j\binom{k-n+j-1}{j}S_{k,n-j,1}^{\star}.
 \]
\end{lem}

\begin{proof}
 Each $\zeta_{\A}(k_1,\dots,k_n)$, where $(k_1,\dots,k_n)\in I_{k,n,1}$,
 can be written as a sum of the values of the form
 $(-1)^j\zeta_{\A}^{\star}(k_1',\dots,k_{n-j}')$
 where $j=0,\dots,n-1$ and $(k_1',\dots,k_{n-j}')\in I_{k,n-j,1}$.
 Moreover, each $(k_1',\dots,k_{n-j}')\in I_{k,n-j,1}$ appears in this manner
 exactly as many times as there are ways of adding $j$ bars
 to the $n-j-1$ existing bars in the gaps in the following sequence of stars,
 in such a way that no bar separates the leftmost two stars
 and no two bars are in the same gap:
 \[
  \underbrace{\fbox{$\star\star$}\cdots\star}_{k_1'}\vert\cdots\vert
  \underbrace{\star\cdots\star}_{k_{n-j}'}
 \]
 Since there are $(k_1'-2)+(k_2'-1)+\dots+(k_{n-j}'-1)=k-n+j-1$ gaps that accept bars,
 the number of ways is $\binom{k-n+j-1}{j}$.
\end{proof}

\begin{lem}[Initial values]\label{lem:sum_zeta_i=1}
 For $k,n\in\Z$ with $1\le n\le k-1$, we have
 \[
  S_{k,n,1}
  =\biggl(1+(-1)^n\binom{k-1}{n-1}\biggr)\frac{B_{p-k}}{k}.
 \]
\end{lem}

\begin{proof}
 By Proposition~\ref{prop:sum_zeta-star} and Lemma~\ref{lem:S_S-star},
 we have
 \begin{align*}
  S_{k,n,1}
  &=\sum_{j=0}^{n-1}(-1)^j\binom{k-n+j-1}{j}
   \biggl((-1)^{n-j}+\binom{k-1}{n-j-1}\biggr)\frac{B_{p-k}}{k}\\
  &=\Biggl((-1)^n\sum_{j=0}^{n-1}\binom{k-n+j-1}{j}
    +\sum_{j=0}^{n-1}(-1)^j\binom{k-n+j-1}{j}\binom{k-1}{n-j-1}\Biggr)\frac{B_{p-k}}{k}.
 \end{align*}

 Recall that $(1-x)^{-m}=\sum_{j=0}^{\infty}\binom{m+j-1}{j}x^j\in\Q[[x]]$
 for $m\in\Z_{\ge1}$.
 Looking at the coefficient of $x^{n-1}$ in the product of
 $(1-x)^{-(k-n)}$ and $(1-x)^{-1}$ gives
 \[
  \sum_{j=0}^{n-1}\binom{(k-n)+j-1}{j}=\binom{(k-n+1)+(n-1)-1}{n-1}=\binom{k-1}{n-1};
 \]
 looking at the coefficient of $x^{n-1}$ in the product of
 $(1+x)^{-(k-n)}$ and $(1+x)^{k-1}$ gives
 \[
  \sum_{j=0}^{n-1}(-1)^j\binom{k-n+j-1}{j}\binom{k-1}{n-j-1}=1.
 \]
 The proof is now complete.
\end{proof}

\begin{prop}\label{prop:sum_zeta}
 For $k,n,i\in\Z$ with $1\le i\le n\le k-1$, we have
 \[
  S_{k,n,i}
  =(-1)^{i-1}\biggl(\binom{k-1}{i-1}+(-1)^n\binom{k-1}{n-i}\biggr)\frac{B_{p-k}}{k}.
 \]
\end{prop}

\begin{proof}
 The proof is by induction on $i$,
 the case $i=1$ being Lemma~\ref{lem:sum_zeta_i=1}.
 Assume that the identity holds for $i$.
 Then Proposition~\ref{prop:recurrence} shows that
 \begin{align*}
  iS_{k,n,i+1}
  &=-(n-i)S_{k,n,i}-(k-n)S_{k,n-1,i}\\
  &=-(n-i)(-1)^{i-1}\biggl(\binom{k-1}{i-1}+(-1)^n\binom{k-1}{n-i}\biggr)\frac{B_{p-k}}{k}\\
  &\qquad-(k-n)(-1)^{i-1}
   \biggl(\binom{k-1}{i-1}+(-1)^{n-1}\binom{k-1}{n-i-1}\biggr)\frac{B_{p-k}}{k}\\
  &=-(-1)^{i-1}\biggl((n-i)\binom{k-1}{i-1}+(k-n+i)(-1)^n\binom{k-1}{n-i-1}\biggr)
   \frac{B_{p-k}}{k}\\
  &\qquad-(-1)^{i-1}
   \biggl((k-n)\binom{k-1}{i-1}+(k-n)(-1)^{n-1}\binom{k-1}{n-i-1}\biggr)\frac{B_{p-k}}{k}\\
  &=(-1)^i\Biggl((k-i)\binom{k-1}{i-1}
    +i(-1)^n\binom{k-1}{n-i-1}\Biggr)
    \frac{B_{p-k}}{k}\\
  &=i(-1)^i\biggl(\binom{k-1}{i}+(-1)^n\binom{k-1}{n-i-1}\biggr)\frac{B_{p-k}}{k}.
 \end{align*}
 Therefore the identity holds for $i+1$ as well and the proof is complete.
\end{proof}

Combining Propositions~\ref{prop:sum_zeta-star} and \ref{prop:sum_zeta},
we have completed the proof of the main theorem (Theorem~\ref{thm:main}).

\section*{Acknowledgements}
The authors would like to thank Masanobu Kaneko and Tatsushi Tanaka
for helpful comments, and Hiroki Kondo for carefully reading the manuscript.

\end{document}